\documentclass[12pt]{amsart}
\pdfoutput=1
\usepackage{amsmath}
\usepackage{amsfonts}
\usepackage{amsthm}
\usepackage{verbatim}
\usepackage[svgnames]{xcolor}
\usepackage{tikz}
\usepackage[pdftex]{hyperref}

\tikzstyle{connector} = [->,thick]
\usetikzlibrary{calc}
\pgfdeclarelayer{edgelayer}
\pgfdeclarelayer{nodelayer}
\pgfsetlayers{edgelayer,nodelayer,main}

\newtheorem{thm}{Theorem}[section]
\newtheorem{lem}[thm]{Lemma}
\newtheorem*{remark*}{Remark}
\newtheorem{remark}{Remark}

\newtheorem{conj}[thm]{Conjecture}

\theoremstyle{definition}

\begin{document}

\title[Tur\'an numbers of Multiple Paths and Equibipartite Forests]{Tur\'an numbers of Multiple Paths and Equibipartite Forests}

\date{\today}
\author[N.~Bushaw]{Neal Bushaw$^\dag$}
\address{$^\dag$Department of Mathematical Sciences, University of Memphis, Memphis, TN 38152 USA. Ph: {\tt +1~901~678~1319}, Fax: {\tt +1~901~678~4481}
}
\email{nobushaw@memphis.edu}

\author[N.~Kettle]{Nathan Kettle$^\ddag$}
\address{$^\ddag$Department of Pure Mathematics and Mathematical Statistics, University of Cambridge, Cambridge CB3 0WB, UK. Ph: {\tt +44~1223~7~64253}}
\email{n.kettle@dpmms.cam.ac.uk}

\begin{abstract}
The \emph{Tur\'an number} of a graph \(H\), \(\textrm{ex}(n,H)\), is the maximum number of edges in any graph on \(n\) vertices which does not contain \(H\) as a subgraph. Let \(P_l\) denote a path on \(l\) vertices, and \(k\cdot P_l\) denote \(k\) vertex-disjoint copies of \(P_l\). We determine \(\mathrm{ex}(n,k\cdot P_3)\) for \(n\) appropriately large, answering in the positive a conjecture of Gorgol. Further, we determine \( \mathrm{ex}\left(n,k\cdot P_l\right)\) for arbitrary \(l\), and \(n\) appropriately large relative to \(k\) and \(l\). We provide some background on the famous Erd\H{o}s-S\'os conjecture, and conditional on its truth we determine \( \mathrm{ex}(n,H)\) when \(H\) is an equibipartite forest, for appropriately large \(n\).
\end{abstract}



\maketitle

\section{Introduction}
Our notation in this paper is standard (see, e.g., \cite{mgt}). Thus \(G\cup H\) denotes the disjoint union of graphs \(G\) and \(H\), and we write \(G+H\) for the join of \(G\) and \(H\), the graph obtained from \(G\cup H\) by adding edges between all vertices of \(G\) and all vertices of \(H\), \(K_t\) for the complete graph on \(t\) vertices, \(E_t\) for the empty graph on \(t\) vertices, and \(M_t\) for a maximal matching on \(t\) vertices; that is, the graph on \(t\) vertices consisting of \(\left\lfloor\frac{t}{2}\right\rfloor\) independent edges. We also take this opportunity to point out that unless explicitly stated, any graph named \(G\) is assumed to be on vertex set \(V=\left[n\right]\) and edge set \(E\); we also make no requirement that the subgraphs we find be induced. 

The Tur\'an number, \( \mathrm{ex}(n,H)\), of a graph \(H\) is the maximum number of edges in a graph on \(n\) vertices which does not contain \(H\) as a subgraph. The problem of determining Tur\'an numbers for assorted graphs traces its history back to 1907, when Mantel (see, e.g., \cite{mgt}) proved that the maximum number of edges in an \(n\)-vertex triangle-free graph is \(\left\lfloor\frac{n^2}{4}\right\rfloor\). In 1940, P\'al Tur\'an  \cite{Turan1,Turan2} proved that the extremal graph avoiding \(K_{r}\) as a subgraph is the complete \((r-1)\)-partite graph on \(n\) vertices which is `as balanced as possible': this is the Tur\'an graph \(T_{r-1}(n)\) . Later, Simonovits \cite{simonovits} showed that for large \(n\), the extremal graph forbidding \(p\cdot K_{r}\) is \(K_{p-1}+T_{r-1}(n-p+1)\). For notation, \(H_{\mathrm{ex}}(n,G)\) will be used to represent a graph on \(n\) vertices with no copy of \(G\) as a subgraph, and exactly \(\mathrm{ex}(n,G)\) edges. We note that in general, the extremal graph(s) may not be unique.

Recently, Gorgol \cite{gorgol} proved upper and lower bounds on the extremal number for forbidding several vertex-disjoint copies of an arbitrary connected graph. We determine this number for paths of length 3 in Section \ref{sec:len3}, longer paths in Section \ref{sec:lenl}, and for forests of equibipartite trees in Section \ref{sec:trees}. We also provide some background on the Erd\H{o}s-S\'os conjecture in Section \ref{sec:erdsos}, as our result for trees is conditional on its validity.

\section{Extremal Numbers for Disjoint Paths}

We start by looking at graphs with no disjoint paths of length three. The extremal case here is slightly different than for longer paths, but the proof introduces the main ideas we shall use in proving the result for all paths, as well as the general tools needed for our results on forests.

\subsection{Paths of length 3}
\label{sec:len3}

Gorgol \cite{gorgol} gave constructions giving the following lower bound regarding paths of length three:
\begin{equation*}
 \mathrm{ex}(n,k\cdot P_3)\geq
\begin{cases} \binom{3k-1}{2}+\left\lfloor\frac{n-3k+1}{2}\right\rfloor\text{, for }3k\leq n<5k-1,\\
\binom{k-1}{2}+(n-k+1)(k-1)+\lfloor\frac{n-k+1}{2}\rfloor\text{, for }n\geq 5k-1.
\end{cases}
\end{equation*}

\begin{remark*} This bound is obtained by noting that for any connected graph \(G\) on \(v\) vertices, and for any positive integers \(n,k\) such that \(n\geq kv\), the graphs \(H_{\mathrm{ex}}(n-kv+1,G)\cup K_{kv-1}\) and \(H_{\mathrm{ex}}(n-k+1,G)+K_{k-1}\) do not contain \(k\) vertex-disjoint copies of \(G\). Applying this to forbidding copies of \(K_3\) and counting the edges in these graphs gives the above bound.\end{remark*}

Gorgol conjectured that this is the correct value of \( \mathrm{ex}(n,k\cdot P_3)\), and proved that this is indeed true for \(k=2,3\). Our first result shows that the second construction is best possible for any \(k\) and large enough \(n\).

\begin{thm}
\label{p3thm}
\( \mathrm{ex}(n,k\cdot P_3)=\binom{k-1}{2}+(n-k+1)(k-1)+\lfloor\frac{n-k+1}{2}\rfloor\), for \(n\geq 7k\).
\end{thm}
There is a unique graph for which this bound is attained, namely \(K_{k-1}+M_{n-k+1}\), as in Figure \ref{fig:p3}. This graph does not contain \(k\) disjoint copies of \(P_3\), since each \(P_3\) must contain at least one vertex from the \((k-1)\)-clique.

\begin{figure}[ht]
\centering\fbox{
\begin{tikzpicture}[thick,scale=0.75]
   \begin{pgfonlayer}{nodelayer}
        \node (t_1) at (0,2){};
        \node (t_2) at (1,2){};
        \node (t_3) at (2,2){};
        \node (t_4) at (4,2){};
        \node(tellip) at (3,2){$\cdots$};
        \node(tnum) at (2,3){$K_{k-1}$};
        \node (b_1) at (-.5,0){};
        \node (b_2) at (.5,0){};
        \node (b_3) at (1.5,0){};
        \node (b_4) at (2.5,0){};
        \node (b_5) at (4.5,0){};
        \node (b_6) at (5.5,0){};
        \node(bellip) at (3.5,0){$\cdots$};
        \node(bnum) at (2,-0.5){$M_{n-k+1}$};
        \foreach \x in {1,...,4} {\fill (t_\x) circle (0.1);}
        \foreach \x in {1,...,6} {\fill (b_\x) circle (0.1);}
    \end{pgfonlayer}
    \begin{pgfonlayer}{edgelayer}

    \end{pgfonlayer}
        \foreach \x/\y in {1,...,4}{\foreach \y in {1,...,6}{\draw (t_\x) to (b_\y);}}
        \draw (b_1) to (b_2); 
        \draw (b_3) to (b_4);
        \draw (b_5) to (b_6);

        \draw (t_1) to [in=155,out=25] (t_2);
        \draw (t_1) to [in=155,out=25] (t_3);
        \draw (t_1) to [in=155,out=25] (t_4);
        \draw (t_2) to [in=155,out=25] (t_3);
        \draw (t_2) to [in=155,out=25] (t_4);
        \draw (t_3) to [in=155,out=25] (t_4);
\end{tikzpicture}}
\caption{}
\label{fig:p3}
\end{figure}

\begin{proof}
We proceed by induction on \(k\). For \(k=1\), we shall use the following easy lemma.

\begin{lem}
\label{p3lem}
If \(G\) is a graph on \(n\) vertices which contains no \(P_3\), then \(G\) contains at most \(\lfloor\frac{n}{2}\rfloor\) edges; that is, \( \mathrm{ex}(n,P_3)\leq\left\lfloor\frac{n}{2}\right\rfloor\).
\end{lem}
\begin{proof} If \(G\) contains no \(P_3\), then no vertex can have degree \(\geq 2\), and so \(G\) consists of independent edges; thus the lemma holds.\end{proof}

For the induction step, suppose \(G\) is a graph on \(n\) vertices, with \(m>\binom{k-1}{2}+(n-k+1)(k-1)+\lfloor\frac{n-k+1}{2}\rfloor\) edges, and containing no \(k\cdot P_3\). The number of edges incident to any \(P_3\) in \(G\) must be at least:

\begin{align*}
m-&\mathrm{ex}(n-3,(k-1)\cdot P_3)\\
&\geq\binom{k-1}{2}+(n-k+1)(k-1)+\left\lfloor\frac{n-k+1}{2}\right\rfloor+1\\
&\hspace{0.5cm}-\binom{k-2}{2}-(n-k-1)(k-2)-\left\lfloor\frac{n-k-1}{2}\right\rfloor\\
&= n+2k-3.\end{align*}

Otherwise, the graph induced by the vertices not on this \(P_3\) contains \((k-1)\cdot P_3\) by induction, showing that \(G\) does contain \(k\cdot P_3\).

By the induction hypothesis we can find \(k-1\) vertex-disjoint copies of \(P_3\) in our graph, and each of these must contain a vertex of degree at least \((n+2k-3)/3\). Otherwise, the total number of edges with an endpoint on this \(P_3\) is smaller than \(n+2k-3\). Taking such a high degree vertex from each \(P_3\) gives us a set \(U\) of \(k-1\) vertices each of degree at least \((n+2k-3)/3\).

Assume that \(G[V\setminus U]\) contains \(P_3\). Then, we can still construct another \(k-1\) copies of \(P_3\), each centered on a vertex from \(U\), as long as each vertex in \(U\) has degree large enough to ensure it is connected to at least two vertices not contained on any of the other \(k-1\) copies of \(P_3\); i.e. if \((n+2k-3)/3\geq 3k-1\), and this is the case when \(n\geq7k\). Therefore \(G[V\setminus U]\) consists of independent edges and isolated vertices, and so \(G\) has at most \(\binom{k-1}{2}+(n-k+1)(k-1)+\lfloor\frac{n-k+1}{2}\rfloor\) edges, a contradiction.
\end{proof}

The above proof gives the extremal graph for \(n\geq 7k\). No construction is known giving a better bound for \(n\geq 5k-1\), and we conjecture that the above example is optimal in this range.

\subsection{Longer Paths}
\label{sec:lenl}
We note at this point that in the proof of Theorem \ref{p3thm}, in order to find a \(P_3\) it was enough to find a vertex of degree two; to find subsequent copies of \(P_3\), it sufficed to find vertices of large degree. To adapt this idea to longer paths, we'll look for sets of vertices with large common neighbourhood. This notion will continue to be an integral part of our proofs, and thus we formalize it here.

\begin{lem}
\label{badlemma}
Let \(G\) be a graph on \(n\) vertices with \(m\) edges, \(t\in\mathbb{N}\), and let \(F_1,F_2\) be arbitrary graphs. Then if \(F_1\cup F_2\not\subseteq G\), any \(F_1\) in \(G\) contains \(t\) vertices with shared neighbourhood of size at least \(n'\geq\frac{m^\prime-\left(n-r\right)\left(t-1\right)}{r-t+1}/\binom{r}{t}\), where \(m^\prime=m- \mathrm{ex}(n-r,F_2)-\binom{r}{2}\), and \(r=|V(F_1)|\).
\end{lem}
\begin{proof}
Assume \(F_1\subseteq G\), say on vertex set \(U\). Since \(G\) contains no \(F_1\cup F_2\), \(G[V\setminus U]\) contains no \(F_2\). Thus \(G[V\setminus U]\) contains at most \( \mathrm{ex}(n-r,F_2)\) edges, and so \(U\) must have at least \(m-\mathrm{ex}(n-r,F_2)-\binom{r}{2}=m^\prime\) edges to \(V\setminus U\). Let \(n_0\) be the number of vertices in \(V\setminus U\) with neighbourhood of size at least \(t\) in \(U\); that is, \(n_0=\left|\left\{v\in V\setminus U:\,\left| N_{U}(v)\right|\geq t\right\}\right|\).

\(U\) has at most \(n_0 r+\left(n-r-n_0\right)\left(t-1\right)\) edges to \(V\setminus U\). Thus \(n_0 r+\left(n-r-n_0\right)\left(t-1\right)\geq m^\prime\), so \(n_0\geq\frac{m^\prime-\left(n-r\right)\left(t-1\right)}{r-t+1}\). Trivially, there are only \(\binom{r}{t}\) subsets of size \(t\) in \(F_1\), and so some subset has shared neighbourhood of size \(n^\prime\geq\frac{m^\prime-\left(n-r\right)\left(t-1\right)}{r-t+1}/\binom{r}{t}\) as claimed.
\end{proof}

The proof of Lemma \ref{p3thm} also required the value of \( \mathrm{ex}(n,P_3)\); for longer paths, we shall use the following result due to Erd\H{o}s and Gallai \cite{EG}.
\begin{thm}
\label{erdosgallai}
For any \(n,l\in\mathbb{N}\), \( \mathrm{ex}(n,P_l)\leq \frac{l-2}{2}n\).
\end{thm}

We note that the bound in Theorem \ref{erdosgallai} is attained by taking disjoint copies of \(K_{l-1}\) as in Figure \ref{fig:erdosgalfig}; this gives a tight result whenever \(n\) is divisible by \(l-1\).
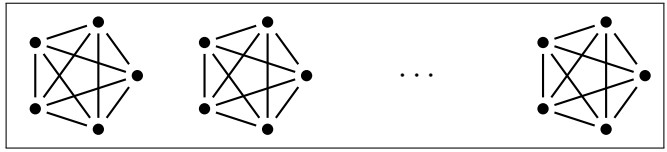
\begin{figure}[ht]
\centering\fbox{
\begin{tikzpicture}[thick,scale=0.75]

    \begin{pgfonlayer}{nodelayer}
        \foreach \x in {1,...,5} {\node (l_\x) at ($(72*\x:1cm)$){}; \node (m_\x) at ($(72*\x:1cm)+(3,0)$){}; \node (r_\x) at ($(72*\x:1cm)+(9,0)$){};}
        \node at (6,0){$\cdots$};
        \foreach \x in {1,...,5} {\fill (l_\x) circle (0.1); \fill (m_\x) circle (0.1); \fill (r_\x) circle (0.1);}
    \end{pgfonlayer}
        \foreach \x/\y in {1/2,1/3,1/4,1/5,2/3,2/4,2/5,3/4,3/5,4/5}{\draw (l_\x) to (l_\y); \draw (m_\x) to (m_\y); \draw (r_\x) to (r_\y);}
\end{tikzpicture}}
\caption{Extremal graph forbidding \(P_6\).}
\label{fig:erdosgalfig}
\end{figure}

We are now ready to prove the main result of this section.
\begin{thm}
\label{longpath}
For \(k\geq 2\), \(l\geq 4\), and \(n\geq 2l+2kl\left(\left\lceil\frac{l}{2}\right\rceil+1\right)\binom{l}{\left\lfloor\frac{l}{2}\right\rfloor}\),
\[ \mathrm{ex}(n,k\cdot P_l)=\binom{k\left\lfloor\frac{l}{2}\right\rfloor-1}{2}+\left(k\left\lfloor\frac{l}{2}\right\rfloor
-1\right)\left(n-k\left\lfloor\frac{l}{2}\right\rfloor+1\right)~+~c_l,\] where \(c_l=1\) if \(l\) is odd, and \(c_l=0\) if \(l\) is even.
\end{thm}

Note that the result above for \(k\cdot P_l\) for \(l\geq4\) does not match the earlier result for \(k\cdot P_3\) in Theorem \ref{p3thm}. 

The extremal graph here is \(G(n,k,l):=K_{t}+E_{n-t}\), with a single edge added to the empty class when \(l\) is odd, and \(t=k\left\lfloor\frac{l}{2}\right\rfloor-1\), as seen in Figures \ref{fig:plodd}, \ref{fig:pleven} respectively.

\begin{figure}[ht]
\centering
\begin{minipage}{0.45\linewidth}
\centering\fbox{
\begin{tikzpicture}[thick,scale=0.75]
    \begin{pgfonlayer}{nodelayer}
        \node (t_1) at (0,2){};
        \node (t_2) at (1,2){};
        \node (t_3) at (2,2){};
        \node (t_4) at (4,2){};
        \node(tellip) at (3,2){$\cdots$};
        \node(tnum) at (2,3){$K_{k\left\lfloor\frac{l}{2}\right\rfloor-1}$};
        \node (b_1) at (-.5,0){};
        \node (b_2) at (.5,0){};
        \node (b_3) at (1.5,0){};
        \node (b_4) at (2.5,0){};
        \node (b_5) at (4.5,0){};
        \node (b_6) at (5.5,0){};
        \node(bellip) at (3.5,0){$\cdots$};
        \node(bnum) at (1.5,-1){$P_2+E_{n-k\left\lfloor\frac{l}{2}\right\rfloor-1}$};
        \foreach \x in {1,...,4} {\fill (t_\x) circle (0.1);}
        \foreach \x in {1,...,6} {\fill (b_\x) circle (0.1);}
    \end{pgfonlayer}
    \begin{pgfonlayer}{edgelayer}

    \end{pgfonlayer}
        \foreach \x/\y in {1,...,4}{\foreach \y in {1,...,6}{\draw (t_\x) to (b_\y);}}
        \draw (b_1) to (b_2); 
        \draw (t_1) to [in=155,out=25] (t_2);
        \draw (t_1) to [in=155,out=25] (t_3);
        \draw (t_1) to [in=155,out=25] (t_4);
        \draw (t_2) to [in=155,out=25] (t_3);
        \draw (t_2) to [in=155,out=25] (t_4);
        \draw (t_3) to [in=155,out=25] (t_4);
\end{tikzpicture}}
\caption{}
\label{fig:plodd}
\end{minipage}
\hspace{0.5cm}
\begin{minipage}{0.45\linewidth}
\centering\fbox{
\begin{tikzpicture}[scale=0.75,thick]
    \begin{pgfonlayer}{nodelayer}
        \node (t_1) at (0,2){};
        \node (t_2) at (1,2){};
    \node (t_3) at (2,2){};
    \node (t_4) at (4,2){};
    \node(tellip) at (3,2){$\cdots$};
    \node(tnum) at (2,3){$K_{k\left\lfloor\frac{l}{2}\right\rfloor-1}$};
    \node (b_1) at (-.5,0){};
    \node (b_2) at (.5,0){};
    \node (b_3) at (1.5,0){};
    \node (b_4) at (2.5,0){};
    \node (b_5) at (4.5,0){};
    \node (b_6) at (5.5,0){};
    \node(bellip) at (3.5,0){$\cdots$};
    \node(bnum) at (2,-1){$E_{n-k\left\lfloor\frac{l}{2}\right\rfloor+1}$};
    \foreach \x in {1,...,4} {\fill (t_\x) circle (0.1);}
    \foreach \x in {1,...,6} {\fill (b_\x) circle (0.1);}
  \end{pgfonlayer}
  \begin{pgfonlayer}{edgelayer}
  \end{pgfonlayer}
    \foreach \x/\y in {1,...,4}{\foreach \y in {1,...,6}{\draw (t_\x) to (b_\y);}}
    \draw (t_1) to [in=155,out=25] (t_2);
    \draw (t_1) to [in=155,out=25] (t_3);
    \draw (t_1) to [in=155,out=25] (t_4);
    \draw (t_2) to [in=155,out=25] (t_3);
    \draw (t_2) to [in=155,out=25] (t_4);
    \draw (t_3) to [in=155,out=25] (t_4);
\end{tikzpicture}}
\caption{}
\label{fig:pleven}
\end{minipage}
\end{figure}

\begin{remark}
We note that for paths of even lengths, the above bound can be proven, and the extremal structure determined, via a paper of Balister, Gy\H{o}ri, Lehel, and Schelp \cite{BGLS} as a consequence of a theorem regarding the maximal number of edges in a connected graph containing no path of some fixed length. One can divide a long path into many short even paths, and this allows one to deduce our Theorem \ref{longpath} from their Theorem 1.3; for odd length paths this result gives a nonoptimal number of edges due to parity issues. This extremal number within connected graphs was also determined earlier by Kopylov in 1977 \cite{Kopylov}, but the approach in the proof given there did not give the extremal structure.
\end{remark}

\begin{proof}
We proceed by induction on k, starting with the base case, \(k=2\).

Let \(G\) be a graph with \(|V|=n\geq 2l+4l\left(\left\lceil\frac{l}{2}\right\rceil+1\right)\binom{l}{\left\lfloor\frac{l}{2}\right\rfloor}\), \(|E(G)|\geq\binom{2\left\lfloor\frac{l}{2}\right\rfloor-1}{2}+\left(2\left\lfloor\frac{l}{2}\right\rfloor
-1\right)\left(n-2\left\lfloor\frac{l}{2}\right\rfloor+1\right)~+~c_l\), and which contains no \(2\cdot P_l\). As \(n\geq l^2\), we have that \(|E(G)|> \mathrm{ex}(n,P_l)\), and so \(G\) contains a \(P_l\) on vertex set \(U\). Using Lemma \ref{badlemma} with \(F_1=P_l\), \(F_2=P_l\), and \(m=\binom{2\left\lfloor\frac{l}{2}\right\rfloor-1}{2}+\left(2\left\lfloor\frac{l}{2}\right\rfloor
-1\right)\left(n-2\left\lfloor\frac{l}{2}\right\rfloor+1\right)~+~c_l\), some elementary simplification shows that any \(P_l\) contained in \(G\) must have at least \(\left\lfloor\frac{l}{2}\right\rfloor\) vertices sharing a neighbourhood of size at least 
\begin{align*}
n^\prime=&\frac{\binom{2\left\lfloor\frac{l}{2}\right\rfloor-1}{2}+\left(2\left\lfloor\frac{l}{2}\right\rfloor-1\right)\left(n-2\left\lfloor\frac{l}{2}\right\rfloor+1\right)}{\left(\left\lceil\frac{l}{2}\right\rceil+1\right)\binom{l}{\left\lfloor\frac{l}{2}\right\rfloor}}\\
&+\frac{c_l-\mathrm{ex}(n-l,P_l)-\binom{l}{2}-(n-l)\left(\left\lfloor\frac{l}{2}\right\rfloor-1\right)}{\left(\left\lceil\frac{l}{2}\right\rceil+1\right)\binom{l}{\left\lfloor\frac{l}{2}\right\rfloor}}\\
\geq&\frac{\binom{2\left\lfloor\frac{l}{2}\right\rfloor-1}{2}+\left(2\left\lfloor\frac{l}{2}\right\rfloor-1\right)\left(n-2\left\lfloor\frac{l}{2}\right\rfloor+1\right)}{\left(\left\lceil\frac{l}{2}\right\rceil+1\right)\binom{l}{\left\lfloor\frac{l}{2}\right\rfloor}}\\
&+\frac{c_l-\left(n-l\right)\left(\frac{l}{2}-1\right)-\binom{l}{2}-\left(n-l\right)\left(\left\lfloor\frac{l}{2}\right\rfloor-1\right)}{\left(\left\lceil\frac{l}{2}\right\rceil+1\right)\binom{l}{\left\lfloor\frac{l}{2}\right\rfloor}}\\
\geq&\frac{\left(1-\frac{c_l}{2}\right)n-l}{\left(\left\lceil\frac{l}{2}\right\rceil+1\right)\binom{l}{\left\lfloor\frac{l}{2}\right\rfloor}}\\
\end{align*}

We now create an \(\left\lfloor\frac{l}{2}\right\rfloor\)-uniform hypergraph \(\mathcal{H}\) with \(V(\mathcal{H})=V(G)\) as follows: for any \(P_l\subseteq G\), we find a subset \(U'\) of \(\left\lfloor\frac{l}{2}\right\rfloor\) vertices with a large common neighbourhood, as above, and add \(U'\) as an edge in \(\mathcal{H}\).

We now flatten this hypergraph to form a simple graph \(G'\) on the same vertex set, with \(uv\in E(G')\) whenever \(u\) and \(v\) are contained in the same hyperedge.

\begin{figure}[ht]
\centering\fbox{
\begin{tikzpicture}[scale=0.75, thick]
  \begin{pgfonlayer}{nodelayer}
    \foreach \x in {1,...,5}{\node (t_\x) at (\x*72:1cm){}; \node (r_\x) at ($(t_\x)+(4.5,0)$){};};
    \node(arr_1) at (0:2cm){};
    \node(arr_2) at (0:3cm){};
\end{pgfonlayer}

 \draw [connector] (arr_1) to (arr_2);
 \draw (r_1) to (r_2);
 \draw (r_2) to (r_3);
 \draw (r_3) to (r_1);
 \draw (r_5) to (r_1);
 \draw (r_5) to (r_2);

\begin{scope}[fill opacity=0.7, draw opacity = 1.0]
 \filldraw[fill=lightgray!80] ($(t_1)+(0.3,0.3)$)
  to[out=90,in=90] ($(t_2) + (-.5,0)$)
  to[out=270,in=90] ($(t_3) + (-0.3,-0.3)$)
  to[out=270,in=0] ($(t_1) + (0,0.5)$);
 \filldraw[fill=lightgray!40] ($(t_5) + (0.3,0.3)$)
  to[out=90,in=0] ($(t_1) + (0,0.5)$)
  to[out=180,in=90] ($(t_2) + (-.3,-.3)$)
  to[out=270,in=270] ($(t_5) + (0.3,0.3)$);
\end{scope}
\foreach \x in {1,...,5} {\fill (t_\x) circle (0.1); \fill (r_\x) circle (0.1);}
\end{tikzpicture}}
\caption{Flattening a Hypergraph.}
\label{fig:flatten}
\end{figure}
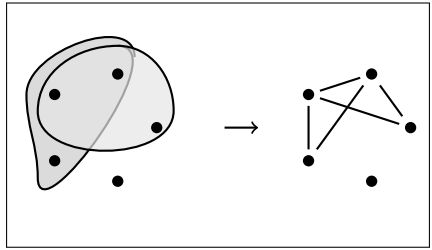

Since vertices adjacent in \(G'\) have large common neighbourhood, a path of length \(\left\lfloor\frac{l}{2}\right\rfloor\) in \(G'\) lets us find a path of length \(l\) in \(G\). More formally, as \(n^\prime\geq 2l\), if \(G^\prime\) contains \(2\cdot P_{\left\lfloor\frac{l}{2}\right\rfloor}\), we can choose distinct common neighbours for each pair of consecutive vertices in these paths, and distinct neighbours for the end vertices, giving us \(2\cdot P_l\) in \(G\). Thus \(G^\prime\) cannot contain \(2\cdot P_{\left\lfloor\frac{l}{2}\right\rfloor}\).

We further note that certainly two disjoint hyperedges in \(\mathcal{H}\) give rise to two such disjoint paths. Thus every pair of edges in \(\mathcal{H}\) intersect; such a hypergraph is called \emph{intersecting}. We will further call a hypergraph \(k\)-\emph{intersecting} if every pair of edges intersect in at least \(k\) vertices.

\begin{lem}\label{smallhgraph}
If there exists \(X\subseteq V(\mathcal{H})\), with \(|X|=t<\left\lfloor\frac{l}{2}\right\rfloor\), and such that \(X\) contains some vertex from each edge in \(\mathcal{H}\), then \(|E(G)|<|E(G(n,2,l))|\).\end{lem}
\begin{proof}
Assume \(X\) is such a set. By the construction of \(\mathcal{H}\), since \(\mathcal{H}[V(\mathcal{H})\setminus X]\) contains no hyperedges, \(G[V(G)\setminus X]\) contains no \(P_l\), and so Theorem \ref{erdosgallai} tells us that \[|E(G)|\leq\binom{t}{2}+t(n-t)+\frac{l-2}{2}(n-t)\leq\left(2\left\lfloor\frac{l}{2}\right\rfloor-\frac{3}{2}\right)n.\]
Recall that 
\begin{align*}
|E\left(G\left(n,2,l\right)\right)|&=\binom{2\left\lfloor\frac{l}{2}\right\rfloor-1}{2}+\left(2\left\lfloor\frac{l}{2}\right\rfloor-1\right)\left(n-2\left\lfloor\frac{l}{2}\right\rfloor+1\right)+c_l\\
&\geq\left(2\left\lfloor\frac{l}{2}\right\rfloor-1\right)n-l^2,
\end{align*}
and so as \(n>2l^2\), \(|E(G)|<\left|E\left(G\left(n,2,l\right)\right)\right|\).\end{proof}

Now, assume we have at least \(2\left\lfloor\frac{l}{2}\right\rfloor\) vertices contained in edges of \(\mathcal{H}\), but without \(2\cdot P_{\left\lfloor\frac{l}{2}\right\rfloor}\) in \(G^\prime\). We claim that no two hyperedges can intersect in only a single vertex.

If \(E_1,E_2\in E(\mathcal{H})\) with \(E_1\cap E_2=\{x\}\), then \(|E_1\cup E_2|=2\left\lfloor\frac{l}{2}\right\rfloor-1\) vertices, and so \(\mathcal{H}\) contains an edge \(E_3\) not contained in their union. We may assume that this edge intersects \(E_1\cup E_2\) outside \(\{x\}\), as if no such edge exists, we are done by Lemma \ref{smallhgraph} applied to the set \(\{x\}\). Without loss of generality, \(E_3\cap E_1\not\subseteq E_1\cap E_2\cup\{x\}\).

Let us consider two cases.

\begin{figure}[ht]
\centering
\begin{minipage}{0.45\linewidth}
\centering\fbox{
\begin{tikzpicture}[scale=0.74,fill opacity=1.0,thick]

  \begin{pgfonlayer}{nodelayer}
    \node(x_1) at (2,1){};
    \node(x_2) at (4,2){};
    \node(x_3) at (4,0){};
    \node(x_4) at (0,0){};
    \node(x_5) at (0,2){};
    \node(x_6) at (2,3){};
    \node(x_7) at (-1.5,3){};
    \node(e_1) at (0.5,1){};
    \node(e_2) at (3.5,1){};
    \node(e_3) at (2,2.25){};
\end{pgfonlayer}
\begin{scope}[fill opacity=0.8]
 \filldraw[fill=lightgray!0] ($(x_4)-(0.3,0.3)$)
  to[out=0,in=270] ($(x_1) + (0.5,0)$)
  to[out=90,in=0] ($(x_5) + (-0.3,0.3)$)
  to[out=180,in=180] ($(x_4) - (0.3,0.3)$);
 \filldraw[fill=lightgray!40] ($(x_3)+(0.3,-0.3)$)
  to[out=180,in=270] ($(x_1) + (-0.5,0)$)
  to[out=90,in=180] ($(x_2) + (0.3,0.3)$)
  to[out=0,in=0] ($(x_3) + (0.3,-0.3)$);
 \filldraw[fill=lightgray!80] ($(x_2)+(0.3,-0.3)$)
  to[out=180,in=0] ($(x_5)+(-0.3,-0.3)$)
  to[out=180,in=180] ($(x_6)+(0,0.5)$)
  to[out=0,in=0] ($(x_2)+(0.3,-0.3)$);
\end{scope}
\foreach \x in {1,...,5,7} {\fill[fill opacity=1.0] (x_\x) circle (0.1);}
\draw (x_7) to (x_5) to (x_4);
\draw (x_1) to (x_3) to (x_2);
\fill (x_1) circle (0.1) node [below] {$x$};
\fill (e_1) node {$e_1$};
\fill (e_2) node {$e_2$};
\fill (e_3) node {$e_3$};

\end{tikzpicture}}
\caption{Case 1.}
\end{minipage}
\hspace{0.5cm}
\begin{minipage}{0.45\linewidth}
\centering\fbox{
\begin{tikzpicture}[scale=0.75,fill opacity=1.0,thick]
  \begin{pgfonlayer}{nodelayer}
    \node(x_1) at (2,1){};
    \node(e_1) at (0.5,1){\(e_1\)};
    \node(e_2) at (3.5,1){\(e_2\)};
    \node(e_3) at (1.5,2.5){\(e_3\)};
    \node(x_2) at (4,2){};
    \node(x_3) at (4,0){};
    \node(x_4) at (0,0){};
    \node(x_5) at (0,2){};
    \node(x_6) at (2,3){};
\end{pgfonlayer}
\begin{scope}[fill opacity =0.5]
 \filldraw[fill=lightgray!0] ($(x_4)-(0.3,0.3)$)
  to[out=0,in=270] ($(x_1) + (0.5,0)$)
  to[out=90,in=0] ($(x_5) + (-0.3,0.3)$)
  to[out=180,in=180] ($(x_4) - (0.3,0.3)$);
 \filldraw[fill=lightgray!40] ($(x_3)+(0.3,-0.3)$)
  to[out=180,in=270] ($(x_1) + (-0.5,0)$)
  to[out=90,in=180] ($(x_2) + (0.3,0.3)$)
  to[out=0,in=0] ($(x_3) + (0.3,-0.3)$);
 \filldraw[fill=lightgray!80] ($(x_1)+(0.3,-0.3)$)
  to[out=45,in=0] ($(x_6)+(0,0.5)$)
  to[out=180,in=90] ($(x_5)+(-0.3,-0.3)$)
  to[out=270,in=225] ($(x_1)+(0.3,-0.3)$);
\end{scope}
\foreach \x in {1,...,6} {\fill (x_\x) circle (0.1);}
\draw (x_6) to (x_5) to (x_4);
\draw (x_1) to (x_3) to (x_2);
\fill (x_1) circle (0.1) node [below] {$x$};
\fill (e_1) node {$e_1$};
\fill (e_2) node {$e_2$};
\fill (e_3) node {$e_3$};
\end{tikzpicture}}
\caption{Case 2.}
\end{minipage}
\end{figure}

\underline{Case 1:} \(E_3\cap \left(E_2\setminus E_1\right)\neq\emptyset\). Then we can find a cycle in \(G^\prime\) through all the vertices in \(E_1\cup E_2\). Since we have at least \(2\left\lfloor\frac{l}{2}\right\rfloor\) vertices in edges of \(G^\prime\), there is another vertex adjacent to this cycle. This gives us a path of length \(2\left\lfloor\frac{l}{2}\right\rfloor\), and so \(G^\prime\) contains \(2\cdot P_{\left\lfloor\frac{l}{2}\right\rfloor}\).

\underline{Case 2:} \(E_3\cap \left(E_2\setminus E_1\right)=\emptyset\). Then there is some \(y\in E_3\setminus\left(E_1\cup E_2\right)\), and so we can form one \(P_{\left\lfloor\frac{l}{2}\right\rfloor}\) in \(\left(E_1\setminus\{x\}\right)\cup\{y\}\) and a disjoint \(P_{\left\lfloor\frac{l}{2}\right\rfloor}\) entirely inside \(E_2\).

\(\mathcal{H}\) is an intersecting hypergraph, with at least \(2\left\lfloor\frac{l}{2}\right\rfloor\) vertices contained in its edges, and no two edges can intersect in a single vertex, and so \(\mathcal{H}\) is 2-intersecting. \(\mathcal{H}\) has nonempty edge set, so pick an edge \(E\), and any vertex in \(x\in E\). Each edge in \(\mathcal{H}\) intersects \(E\) in at least two vertices, so any edge in \(\mathcal{H}\) intersects \(E\setminus\{x\}\), a set of size \(\left\lfloor\frac{l}{2}\right\rfloor-1\). We have already ruled out such a set of vertices in Lemma \ref{smallhgraph}.

We now know \(G\) contains a set \(A\) of vertices, \(\left|A\right|\leq2\left\lfloor\frac{l}{2}\right\rfloor-1\), with the property that any \(P_l\) in \(G\) contains at least \(\left\lfloor\frac{l}{2}\right\rfloor\) vertices from \(A\). We define three more sets of vertices as follows:
\begin{align*}
B&=\left\{x\in G\setminus A\,|\,d_{A}(x)\geq \left\lfloor\frac{l}{2}\right\rfloor\right\},\\
C&=\left\{x\in G\setminus A\,|\,\left\lfloor\frac{l}{2}\right\rfloor>d_{A}(x)>0\right\},\\
D&=\left\{x\in G\setminus A\,|\,d_{A}(x)=0\right\}.
\end{align*}
Certainly \(D\) can contain no \(P_l\), since every \(P_l\) meets \(A\). Thus the number of edges entirely within \(D\) is at most \(\frac{l-2}{2}|D|\) by Theorem \ref{erdosgallai}.

We now claim that every vertex \(x\in B\cup C\) is the end vertex of a \(P_l\) in \(G\), with alternate vertices in \(A\), which also misses any given \(y_1,y_2\in B\cup C\). Since \(x\) is adjacent to some \(y\in A\), and \(y\) is contained in some hyperedge \(E\), as long as \(n^\prime>\left|A\right|+\left\lfloor\frac{l}{2}\right\rfloor+2\), we can find \(\left\lfloor\frac{l}{2}\right\rfloor\) vertices in \(\left(B\cup C\right)\setminus\{x\}\) adjacent to all vertices in \(E\), allowing us to find such a \(P_l\).

\begin{figure}[ht]
\centering
\begin{minipage}{0.45\linewidth}
\centering\fbox{
\begin{tikzpicture}[scale=0.65,thick]
  \begin{pgfonlayer}{nodelayer}
   \node (a) at (4,3){\(A\)};
   \node (b) at (1,0){\(B\)};
   \node (c) at (4,0){\(C\)};
   \node (d) at (7,0){\(D\)};
   \node (u) at (7,0.5){};
   \node (v) at (4.5,0.5){};
   \node (w) at (3.5,0.5){};
   \foreach \x in {0,...,2}{\node (a_\x) at ($(3.5,3)-(0.5*\x,0)$){}; \node (b_\x) at ($(1.5,0.5)-(0.5*\x,0)$){};}
   \begin{scope}[fill opacity = 0.5]
   \filldraw[fill=lightgray] (a) ellipse (2cm and 1cm);
   \filldraw[fill=lightgray] (b) circle (1cm);
   \filldraw[fill=lightgray] (c) circle (1cm);
   \filldraw[fill=lightgray] (d) circle (1cm);
   \end{scope}
\end{pgfonlayer}
\fill (u) circle (0.1);
\fill (v) circle (0.1);
\fill (w) circle (0.1);
\foreach \x in {0,...,2}{\fill (a_\x) circle (0.1); \fill (b_\x) circle (0.1); \draw (a_\x) to (b_\x);}
\draw (b_0) to (a_1);
\draw (b_1) to (a_2);
\draw (v) to (a_0);
\draw (u) to[out=195,in=-15] (v);
\draw (u) to[out=195,in=-15] (w);
\end{tikzpicture}}
\caption{}
\label{fig:ABC1}
\end{minipage}
\hspace{0.5cm}
\begin{minipage}{.45\linewidth}
\centering\fbox{
\begin{tikzpicture}[scale=0.65,thick]
  \begin{pgfonlayer}{nodelayer}
   \node (a) at (4,3){\(A\)};
   \node (b) at (1,0){\(B\)};
   \node (c) at (4,0){\(C\)};
   \node (d) at (7,0){\(D\)};
   \node (b_5) at (0.5,-0.5){};
   \foreach \x in {0,...,2}{\node (a_\x) at ($(3.5,3)-(0.5*\x,0)$){}; \node (b_\x) at ($(1.5,0.5)-(0.5*\x,0)$){};}
   \begin{scope}[fill opacity = 0.5]
   \filldraw[fill=lightgray] (a) ellipse (2cm and 1cm);
   \filldraw[fill=lightgray] (b) circle (1cm);
   \filldraw[fill=lightgray] (c) circle (1cm);
   \filldraw[fill=lightgray] (d) circle (1cm);
   \end{scope}
  \end{pgfonlayer}
 \fill (b_5) circle (0.1);
 \foreach \x in {0,...,2}{\fill (a_\x) circle (0.1); \fill (b_\x) circle (0.1); \draw (a_\x) to (b_\x);}
 \draw (b_0) to (a_1);
 \draw (b_1) to (a_2);
 \draw (b_2) to (b_5);
\end{tikzpicture}}
\caption{}
\label{fig:ABC2}
\end{minipage}
\end{figure}

Further, no vertex in \(D\) can have degree more than 1 to \(B\cup C\); assume \(uv,uw\) are both edges with \(u\in D\), and \(v,w\in B\cup C\). We can find a \(P_l\) leaving \(v\), that misses \(w\), with alternate vertices in \(A\). This gives a \(P_l\) starting at \(w\) with only \(\left\lfloor\frac{l}{2}\right\rfloor-1\) vertices from \(A\), as in Figure \ref{fig:ABC1}. A vertex in \(C\) with degree 2 to \(B\cup C\) allows us to create a path in the same way, so our graph contains none of these.

Similarly, if \(l\) is even, an edge inside \(B\) allows us to create a \(P_l\) using only \(\left\lfloor\frac{l}{2}\right\rfloor-1\) vertices from \(A\), as in Figure \ref{fig:ABC2}, so in this case \(B\) must be empty. If \(l\) is odd, two edges in \(B\) allow us to create a \(P_l\) with only \(\left\lfloor\frac{l}{2}\right\rfloor-1\) vertices from \(A\), but a \emph{single edge} does not; this is where the \(c_l\) in the theorem arises.

We've now counted edges between \(B\) and \(C\) and between \(B\cup C\) and \(D\) respectively, and counted the edges inside each of \(B\), \(C\), and \(D\). We can use the degree conditions in their definitions to count edges from \(A\) to \(B\), \(C\), and \(D\). Putting these together, we see that

\begin{align*}
|E(G)|&\leq\binom{2\left\lfloor\frac{l}{2}\right\rfloor-1}{2}+\left(2\left\lfloor\frac{l}{2}\right\rfloor-1\right)\left(n-\left(2\left\lfloor\frac{l}{2}\right\rfloor-1\right)-|C|-|D|\right)\\
&\quad+\left(1+\left\lfloor\frac{l}{2}\right\rfloor-1\right)|C|+\left(1+\frac{l-2}{2}\right)|D|+c_l,\\
&\leq\binom{2\left\lfloor\frac{l}{2}\right\rfloor-1}{2}+\left(2\left\lfloor\frac{l}{2}\right\rfloor-1\right)\left(n-\left(2\left\lfloor\frac{l}{2}\right\rfloor-1\right)\right)\\
&\quad+\left(1-\left\lfloor\frac{l}{2}\right\rfloor\right)|C|+\left(\frac{l-2}{2}-2\left\lfloor\frac{l}{2}\right\rfloor+2\right)|D|+c_l.
\end{align*}

Since the coefficients of \(|C|\) and \(|D|\) above are negative, \(|E(G)|\) is maximized when \(C\) and \(D\) are empty; this gives our bound on \(|E(G)|\) as claimed. Further, since \(C\) and \(D\) must be empty to attain this bound it also shows that the extremal graph is \(G(n,2,l)=K_{2\left\lfloor\frac{l}{2}\right\rfloor-1}+E_{n-2\left\lfloor\frac{l}{2}\right\rfloor+1}\) with an extra edge in the empty class for odd \(l\), as claimed.

We have now established the base case \(k=2\). Somewhat surprisingly, the inductive step is easy to show.

Let \(G\) be a graph on \(n\) vertices with
\[m\geq\binom{k\left\lfloor\frac{l}{2}\right\rfloor-1}{2}+\left(k\left\lfloor\frac{l}{2}\right\rfloor
-1\right)\left(n-k\left\lfloor\frac{l}{2}\right\rfloor+1\right)~+~c_l\] edges, not containing \(k\cdot P_l\).

This graph does contain a \(P_l\), and from Lemma \ref{badlemma}, we can find \(\left\lfloor\frac{l}{2}\right\rfloor\) vertices with shared neighbourhood of size at least
\begin{align*}
n^\prime=&\frac{\binom{k\left\lfloor\frac{l}{2}\right\rfloor-1}{2}+\left(k\left\lfloor\frac{l}{2}\right\rfloor-1\right)\left(n-k\left\lfloor\frac{l}{2}\right\rfloor+1\right)}{\left(\left\lceil\frac{l}{2}\right\rceil+1\right)\binom{l}{\left\lfloor\frac{l}{2}\right\rfloor}}\\
&+\frac{c_l-\mathrm{ex}\left(n-l,\left(k-1\right)\cdot P_l\right)-\binom{l}{2}-\left(n-l\right)\left(\left\lfloor\frac{l}{2}\right\rfloor-1\right)}{\left(\left\lceil\frac{l}{2}\right\rceil+1\right)\binom{l}{\left\lfloor\frac{l}{2}\right\rfloor}}\\
=&\frac{\binom{k\left\lfloor\frac{l}{2}\right\rfloor-1}{2}+\left(k\left\lfloor\frac{l}{2}\right\rfloor-1\right)\left(n-k\left\lfloor\frac{l}{2}\right\rfloor+1\right)}{\left(\left\lceil\frac{l}{2}\right\rceil+1\right)\binom{l}{\left\lfloor\frac{l}{2}\right\rfloor}}\\
&+\frac{c_l-\binom{\left(k-1\right)\left\lfloor\frac{l}{2}\right\rfloor-1}{2}-\left(\left(k-1\right)\left\lfloor\frac{l}{2}\right\rfloor-1\right)\left(n-l-\left(k-1\right)\left\lfloor\frac{l}{2}\right\rfloor+1\right)}{\left(\left\lceil\frac{l}{2}\right\rceil+1\right)\binom{l}{\left\lfloor\frac{l}{2}\right\rfloor}}\\
&+\frac{-c_l-\binom{l}{2}-\left(n-l\right)\left(\left\lfloor\frac{l}{2}\right\rfloor-1\right)}{\left(\left\lceil\frac{l}{2}\right\rceil+1\right)\binom{l}{\left\lfloor\frac{l}{2}\right\rfloor}}\\
=&\frac{n+k\left\lfloor\frac{l}{2}\right\rfloor^2-\frac{3}{2}\left\lfloor\frac{l}{2}\right\rfloor^2+c_lk\left\lfloor\frac{l}{2}\right\rfloor-\frac{5+2c_l}{2}\left\lfloor\frac{l}{2}\right\rfloor-2c_l}{\left(\left\lceil\frac{l}{2}\right\rceil+1\right)\binom{l}{\left\lfloor\frac{l}{2}\right\rfloor}}\\
\geq&\frac{n-l}{\left(\left\lceil\frac{l}{2}\right\rceil+1\right)\binom{l}{\left\lfloor\frac{l}{2}\right\rfloor}}
\end{align*}
The second inequality is valid since \(n-l\geq2l+2l(k-1)\left(\left\lceil\frac{l}{2}\right\rceil+1\right)\binom{l}{\left\lfloor\frac{l}{2}\right\rfloor}\).
Write \(U\) for the set of vertices from Lemma \ref{badlemma}. Then \(G[V\setminus U]\) is a graph on \(n-\left\lfloor\frac{l}{2}\right\rfloor\) vertices and at least \(\mathrm{ex}(n-\left\lfloor\frac{l}{2}\right\rfloor,(k-1)\cdot P_l)\) edges. If we can find \((k-1)\cdot P_l\), then since \(n^\prime\geq kl\), we can find another \(P_l\) in \(G\) disjoint from these \(k-1\). Therefore there cannot be \(k-1\) disjoint copies of \(P_l\) in \(G[V\setminus U]\), so by the inductive hypothesis, \(G[V\setminus U]=G(n-\left\lfloor\frac{l}{2}\right\rfloor,k-1,l)\). Thus \(G=G(n,k,l)\).
\end{proof}

The above proof shows that our construction is optimal for \(n~=~\operatorname{O}(kl^22^l)\). We conjecture that this construction is optimal for \(n=\operatorname{O}(kl)\). We also note a comparison between Theorem \ref{longpath} for even paths and Theorem \ref{erdosgallai}: certainly if one forbids \(k\cdot P_{2l}\), then one is also forbidding \(P_{2kl}\). Thus an easy upper bound on \( \mathrm{ex}(n,k\cdot P_{2l})\) is \( \mathrm{ex}(n,P_{2kl})\). The difference between this bound and the precise result established above is relatively small, \((kl-1)(\frac{kl}{2})\). In particular, it is not dependent on \(n\) for fixed \(k\) and \(l\), despite the significant difference between the extremal graphs.

\section{Trees}
Throughout the following section, we need an analogue of Lemma \ref{p3lem} as a starting point. For longer paths, we used the Erd\H{o}s-Gallai result, Lemma \ref{erdosgallai}. The analogous result for trees is known as the Erd\H{o}s-S\'os Conjecture.
\subsection{The Erd\H{o}s-S\'os Conjecture}
\label{sec:erdsos}
We note that a path can be viewed as an extreme kind of tree - \(l-2\) vertices have degree two, and the two leaves of course have degree one. The opposite extreme is the star - one central vertex of degree \(l-1\), and the other \(k-1\) vertices are leaves. For both examples, it is easily seen that \( \mathrm{ex}(n,G)=\frac{l-2}{2}n\). Legend has it that Vera T. S\'os presented the proofs of these two results to her graph theory class in Budapest in 1962, and left the following conjecture as a homework problem; by now, this is known as the notoriously difficult Erd\H{o}s-S\'os Conjecture.

\begin{conj}\emph{(Erd\H{o}s-S\'os Conjecture)}
\label{erdossosconj}
For any tree \(T\) on \(l\) vertices, \( \mathrm{ex}(n,T)=\frac{l-2}{2}n\).
\end{conj}

In 2008, a proof of the conjecture was announced for very large trees by Ajtai, Koml\'os, Simonovits, and Szemer\'edi. For small trees, however, the conjecture is mostly open. There is a sequence of results in the direction of the full theorem for smaller trees. We present a representative sample of these results here, which is certainly only the tip of the iceberg. Many more partial results related to the Erd\H{o}s-S\'os Conjecture exist; see for example \cite{dobson2},\cite{wozniak}. The first result here establishes the conjecture for graphs of large girth and is due to Dobson \cite{dobson}.

\begin{thm}
If \(T\) is a tree on \(l\) vertices, and \(G\) is a graph with girth at least five and minimum degree \(\delta\geq\frac{l}{2}\), then \(G\) contains \(T\). Thus Conjecture \ref{erdossosconj} holds for graphs of girth at least 5.
\end{thm}

Similarly, Sacl\'e and Wo\'zniak \cite{saclewozniak} proved that whenever \(G\) is a graph with at least \(\frac{l-2}{2}n\) edges and no \(C_4\), \(G\) contains any tree on \(l\) vertices. In 2005, McLennan \cite{mclennan} proved the Erd\H{o}s-S\'os bound for trees of diameter at most four.

The Erd\H{o}s-S\'os Conjecture has also been proven for caterpillars; this result is attributed to Perles in \cite{mp}.  Later, Sidorenko \cite{sidorenko} showed that the Erd\H{o}s-S\'os Conjecture holds for trees of order \(l\) containing a vertex which is the parent of at least \(\frac{l-1}{2}\) leaves.

\subsection{Forests of Equibipartite Trees}
\label{sec:trees}
Our proof of Theorem \ref{longpath} can be adapted to work on a significantly larger class of graphs. A key element of our proof was finding a set of vertices which intersected every long path in at least half its vertices. This continues to be an essential idea, and thus we restrict ourselves to trees which have the same number of vertices in each vertex class, when viewed as a bipartite graph. We call such trees \emph{equibipartite}, and a forest in which each component is an equibipartite tree is called an equibipartite forest. Clearly any equibipartite tree or equibipartite forest has an even number of vertices.

If we allow ourselves the considerable benefit of assuming that Erd\H{o}s-S\'os holds for all equibipartite trees, we can determine the extremal number for any equibipartite forest, for large \(n\). There is a slight difference in the extremal number and the structure of the extremal graph depending on whether the forest admits a perfect matching.

\begin{thm}
\label{treethm}
Let \(H\) be an equibipartite forest on \(2l\) vertices which is comprised of at least two trees. If the Erd\H{o}s-S\'os Conjecture holds, then for \(n\geq3l^2+32l^5\binom{2l}{l}\),
\[\mathrm{ex}(n,H)=\begin{cases} \binom{l-1}{2}+(l-1)(n-l+1) \textrm{, if }H\textrm{ admits a perfect matching}\\ (l-1)(n-l+1) \textrm{ otherwise.}\end{cases}\]
\end{thm}

\begin{figure}[ht]
\centering
\begin{minipage}{0.45\linewidth}
\centering\fbox{
\begin{tikzpicture}[scale=0.74,fill opacity=1.0,thick]

        \node (t_1) at (0,2){};
        \node (t_2) at (1,2){};
    \node (t_3) at (2,2){};
    \node (t_4) at (4,2){};
    \node(tellip) at (3,2){$\cdots$};
    \node(tnum) at (2,3){$K_{l-1}$};
    \node (b_1) at (-.5,0){};
    \node (b_2) at (.5,0){};
    \node (b_3) at (1.5,0){};
    \node (b_4) at (2.5,0){};
    \node (b_5) at (4.5,0){};
    \node (b_6) at (5.5,0){};
    \node(bellip) at (3.5,0){$\cdots$};
    \node(bnum) at (2,-1){$E_{n-l+1}$};
    \foreach \x in {1,...,4} {\fill (t_\x) circle (0.1);}
    \foreach \x in {1,...,6} {\fill (b_\x) circle (0.1);}
    \foreach \x/\y in {1,...,4}{\foreach \y in {1,...,6}{\draw (t_\x) to (b_\y);}}
    \draw (t_1) to [in=155,out=25] (t_2);
    \draw (t_1) to [in=155,out=25] (t_3);
    \draw (t_1) to [in=155,out=25] (t_4);
    \draw (t_2) to [in=155,out=25] (t_3);
    \draw (t_2) to [in=155,out=25] (t_4);
    \draw (t_3) to [in=155,out=25] (t_4);

\end{tikzpicture}}
\caption{}
\label{fig:experfmatch}
\end{minipage}
\hspace{0.5cm}
\begin{minipage}{0.45\linewidth}
\centering\fbox{
\begin{tikzpicture}[scale=0.75,fill opacity=1.0,thick]
 
        \node (t_1) at (0,2){};
        \node (t_2) at (1,2){};
    \node (t_3) at (2,2){};
    \node (t_4) at (4,2){};
    \node(tellip) at (3,2){$\cdots$};
    \node(tnum) at (2,3){$E_{l-1}$};
    \node (b_1) at (-.5,0){};
    \node (b_2) at (.5,0){};
    \node (b_3) at (1.5,0){};
    \node (b_4) at (2.5,0){};
    \node (b_5) at (4.5,0){};
    \node (b_6) at (5.5,0){};
    \node(bellip) at (3.5,0){$\cdots$};
    \node(bnum) at (2,-1){$E_{n-l+1}$};
    \foreach \x in {1,...,4} {\fill (t_\x) circle (0.1);}
    \foreach \x in {1,...,6} {\fill (b_\x) circle (0.1);}
  \begin{pgfonlayer}{edgelayer}
  \end{pgfonlayer}
    \foreach \x/\y in {1,...,4}{\foreach \y in {1,...,6}{\draw (t_\x) to (b_\y);}}
\end{tikzpicture}}
\caption{}
\label{fig:exnoperfmatch}
\end{minipage}
\end{figure}

\begin{remark*}
The extremal graphs here are \(K_{l-1}+E_{n-l+1}\) for any forest with a perfect matching, and \(E_{l-1}+E_{n-l+1}\) for any forest with no perfect matching, as in Figures \ref{fig:experfmatch}, \ref{fig:exnoperfmatch}. To prove the eventual extremal number for equibipartite trees as in Theorem \ref{treethm}, we do not need the full strength of the Erd\H{o}s-S\'os Conjecture; we only need that the Erd\H{o}s-S\'os Conjecture is true for the trees appearing in the forest \(H\).  In fact, it suffices to know that \(\mathrm{ex}(n,T)=\frac{|T|-2}{2}n+o(n)\) for any of the equibipartite trees \(T\subseteq H\). In this case, however, the bound on \(n\) for which the result holds is much worse. We also note that again in the statement of the theorem we have suppressed lower order terms in the lower bound on \(n\); here the lower order terms from the proof are unnecessarily complicated, and we leave them out.
\end{remark*}

\begin{figure}[ht]
\centering
\begin{minipage}{0.45\linewidth}
\centering\fbox{
\begin{tikzpicture}[scale=0.65,thick]
  \begin{pgfonlayer}{nodelayer}
   \node (arr_1) at (3.5, 1.5){};
   \node (arr_2) at (4.5, 1.5){};
   \node (l_1) at (2,3){};
   \node (l_x) at (2,3.5){\(x\)};
   \node (l_2) at (1,2){};
   \node (l_y) at (1,2.5){\(y\)};
   \node (l_3) at (3,2){};
   \node (l_4) at (2,1){};
   \node (l_5) at (1,0){};
   \node (l_6) at (0,1){};

   \node (r_1) at (6,2){};
   \node (r_x) at (6,2.5){\(x\)};
   \node (r_2) at (7,2){};
   \node (r_y) at (7,2.5){\(y\)};
   \node (r_3) at (5,1){};
   \node (r_4) at (6,1){};
   \node (r_5) at (7,1){};
   \node (r_6) at (8,1){};

\foreach \x in {1,...,6}{\fill (l_\x) circle (0.1); \fill (r_\x) circle (0.1);}
   \begin{scope}[fill opacity = 0.5]
   \end{scope}
  \end{pgfonlayer}
 \draw [connector] (arr_1) to (arr_2);
	\draw (l_1) to (l_2);
	\draw (l_1) to (l_3);
	\draw (l_2) to (l_4);
	\draw (l_2) to (l_6);
	\draw (l_5) to (l_6);

	\draw (r_1) to (r_2);
	\draw (r_1) to (r_3);
	\draw (r_2) to (r_4);
	\draw (r_2) to (r_6);
	\draw (r_5) to (r_6);
\end{tikzpicture}}
\caption{}
\label{fig:perfmatch}
\end{minipage}
\hspace{0.5cm}
\begin{minipage}{.45\linewidth}
\centering\fbox{
\begin{tikzpicture}[scale=0.65,thick]
  \begin{pgfonlayer}{nodelayer}
   \node (arr_1) at (3.5, 1){};
   \node (arr_2) at (4.5, 1){};
   \node (l_1) at (2,2){};
   \node (l_x) at (2,2.5){\(x\)};
   \node (l_2) at (1,1){};
   \node (l_y) at (1,1.5){\(y\)};
   \node (l_3) at (3,1){};
   \node (l_4) at (2,1){};
   \node (l_5) at (0,0){};
   \node (l_6) at (2,0){};

   \node (r_1) at (6,1.5){};
   \node (r_x) at (6,2){\(x\)};
   \node (r_2) at (7,1.5){};
   \node (r_y) at (7,2){\(y\)};
   \node (r_3) at (5,.5){};
   \node (r_4) at (6,.5){};
   \node (r_5) at (7,.5){};
   \node (r_6) at (8,.5){};

\foreach \x in {1,...,6}{\fill (l_\x) circle (0.1); \fill (r_\x) circle (0.1);}
  \end{pgfonlayer}
 \draw [connector] (arr_1) to (arr_2);
	\draw (l_1) to (l_2);
	\draw (l_1) to (l_3);
	\draw (l_1) to (l_4);
	\draw (l_2) to (l_5);
	\draw (l_2) to (l_6);

	\draw (r_1) to (r_2);
	\draw (r_1) to (r_3);
	\draw (r_1) to (r_4);
	\draw (r_2) to (r_5);
	\draw (r_2) to (r_6);
\end{tikzpicture}}
\caption{}
\label{fig:noperfmatch}
\end{minipage}
\end{figure}

\begin{lem}\label{PerfLem} Let \(H\) be a equibipartite tree on \(2l\) vertices. If \(H\) contains a perfect matching, then every partition of \(V(H)\) into two classes of different sizes is such that the larger class induces at least one edge.\end{lem}

\begin{proof}
If \(H\) contains a perfect matching, \(M\subseteq E(H)\), then for any partition of \(V(H)\) into nonequal classes, \(|V_1|<|V_2|\), the number of edges in \(M\) which meet \(V_1\) is at most \(|V_1|<l\), and so some edge lies inside \(V_2\).\end{proof}

\begin{lem}\label{noPerfLem} Let \(H\) be a equibipartite tree on \(2l\) vertices. If \(H\) does not contain a perfect matching, then there exists a partition of \(V(H)\) into two classes of different sizes such that the larger class induces no edges and the smaller class induces exactly one edge.\end{lem}

\begin{proof} Consider \(H\) as a bipartite graph with bipartition \(V(H)=(A,B)\). Since \(H\) contains no perfect matching, there is a set \(S\subseteq A\) for which Hall's condition (see, e.g., \cite{mgt}) fails. If we take \(S\) minimal, then \(H[S\cup N(S)]\) is connected, as otherwise one of its components would fail Hall's condition. Consider \(H[(A\setminus(S))\cup (B\setminus N(S))]\). Each component of this graph is joined to \(N(S)\) by a single edge. Since the union of these components has larger intersection with \(B\) than with \(A\), at least one of the components does. Let \(C\) be such a component, and let \(xy\) be the unique edge between \(C\) and \(N(S)\), with \(x\in C\) and \(y\in N(S)\).

Consider the partition \((C,V(H)\setminus{C})\). Then taking the set of vertices \(V_{x,y}\) which are in the same bipartite class as \(x\) in \(C\) or in the same bipartite class as \(y\) in \(V(H)\setminus C\) as one class of our new partition, and \(V(H)\setminus V_{x,y}\) as the other forms a partition of \(V(H)\) with exactly one edge in \(V_{x,y}\), and none in \(V(H)\setminus V_{x,y}\).

Since our tree is equibipartite, \(|V_{x,y}\cap (V(H)\setminus C)|+|\left(V(H)\setminus V_{x,y}\right)\cap C|=l\). By our definition of \(C\), \(|V_{x,y}\cap C|<|\left(V(H)\setminus V_{x,y}\right)\cap C|\), so we have \(|V_{x,y}|=|V_{x,y}\cap C|+|V_{x,y}\cap (V(H)\setminus C)|<|V_{x,y}\cap C|<|\left(V(H)\setminus V_{x,y}\right)\cap C|=l\), as required.

\end{proof}
See Figures \ref{fig:perfmatch}, \ref{fig:noperfmatch} for an example partition of a trees with and without a perfect matching, respectively.

\begin{proof}[Proof of Theorem \ref{treethm}]
Let \(H\) have components \(H_1,H_2,\ldots,H_k\), each on \(2l_1, 2l_2, \dots, 2l_k\) vertices respectively, and \(G\) be a graph on \(n\) vertices with \(m\) edges which does not contain \(H\), and with \(m\geq(l-1)(n-l+1)\). Without loss of generality, \(l_1\leq l_i\), for each \(i\). For notational ease, we also define \(H^\prime=H_2\cup\ldots\cup H_k\) and \(l^\prime=\frac{1}{2}|H^\prime|=l-l_1\).


As \(n\geq l^2\), \(m\geq\mathrm{ex}(n,H^\prime)\) by induction (or Erd\H{o}s-S\'os, if \(H^\prime\) is a tree), and so we can find a copy of \(H^\prime\subseteq G\). As in the proof of Lemma \ref{badlemma}, for any copy of \(H^\prime\) we can bound from below the size of the set \(E'\) of edges between \(H^\prime\) and \(G\setminus H^\prime\) by \(m-\binom{2l^\prime}{2}-\textrm{ex}\left(n-2l^\prime,H_1\right)\). By the Erd\H{o}s-S\'os Conjecture, this is at least \(\left(l-1\right)\left(n-l+1\right)-\binom{2l^\prime}{2}-\left(n-2l^\prime\right)\left(l_1-1\right)\geq l^\prime n-3l^2\).

Consider the set of vertices \(X=\left\{v\in G\setminus H^\prime: \left|N(v)\cap H^\prime\right|\geq l^\prime\right\}\). Then
\[2l^\prime\left|X\right|+\left(l^\prime-1\right)\left(n-2l^\prime-\left|X\right|\right)\geq\left|E'\right|\geq l^\prime n-3l^2.\]

Thus \(\left|X\right|\geq\frac{n-3l^2}{l^\prime+1}\). As there are only \(\binom{2l^\prime}{l^\prime}\) sets of \(l^\prime\) vertices in \(H^\prime\), we can find a set \(A\) of \(l^\prime\) vertices in \(H^\prime\) with at least \(n^\prime=\frac{n-3l^2}{\left(l^\prime+1\right)\binom{2l^\prime}{l^\prime}}\) common neighbours.

Interchangine the roles of \(H_1\) and \(H^\prime\), for any \(H_1\) we similarly bound from below the size of the set \(E_1\) of edges between \(H_1\) and \(G\setminus H_1\) by \(m-\binom{2l_1}{2}-\textrm{ex}\left(n-2l_1,H^\prime\right)\). Note that \(n-2l_1\) is much larger than needed in the condition of the inductive hypothesis, and so
\begin{align}
|E_1|\geq&\left(l-1\right)\left(n-l+1\right)-\binom{2l_1}{2}-\left(n-l^\prime-2l_1+1\right)\left(l^\prime-1\right)-\binom{l^\prime-1}{2}\notag\\
&\geq l_1 n-3l^2.\label{H1s}
\end{align}

With this in mind, we define the following two sets of vertices:


\begin{align*}
B&=\left\{w\in G|w\not\in A\textrm{ and } d_G(w)\geq\frac{n-3l^2}{l_1+1}\right\}\\
C&=G\setminus(A\cup B)
\end{align*}

Now, any copy of \(H_1\) in \(G\) must contain at least \(l_1\) vertices from \(A\cup B\), as otherwise the total degree of vertices on \(H_1\) is less than \(\left(l_1+1\right)\frac{n-3l^2}{l_1+1}+\left(l_1-1\right)n\), contradicting (\ref{H1s}) above.

As a rough bound on the number of edges in \(G\), we note that if \(G\) contained more than \(2ln\) edges,  we can find a copy of \(H^\prime\) by induction (or by the Erd\H{o}s-S\'os Conjecture if \(H^\prime\) is a single tree). Removing this copy of \(H^\prime\) leaves a graph on \(n-2l^\prime\) vertices with more than \(2l_1n\geq2l_1\left(n-2l^\prime\right)\) edges, since each vertex is of course adjacent to at most \(n\) edges. Again by Conjecture \ref{erdossosconj}, we can find a copy of \(H_1\). Thus our graph can have at most \(2ln\) edges.

This means that for any \(c>0\), there are at most \(\frac{4ln}{c}\) vertices of degree at least \(c\). Choosing \(c=\frac{8ln}{n^\prime}\leq\frac{8l\left(l^\prime+1\right)\binom{2l^\prime}{l^\prime}n}{n-3l^2}\), there are at least \(\frac{n^\prime}{2}\) common neighbours of \(A\) with degree at most \(c\). Since \(n\geq 6l^2\),  \(c\leq16l\left(l^\prime+1\right)\binom{2l^\prime}{l^\prime}\). Then since \(\frac{n^\prime}{2}\gg l^\prime\), we can find a copy of \(H^\prime\) with \(l^\prime\) vertices in \(A\) and the other \(l^\prime\) vertices having degree at most \(c\).

Since this copy of \(H^\prime\) is incident to at least \(l^\prime n-3l^2\) edges, any vertex in \(A\) has degree at least
\begin{eqnarray}\label{cprime}
&l^\prime n-3l^2-l^\prime c-\left(l^\prime-1\right)\left(n-1\right)\\
&\geq n-3l^2-l^\prime c\notag\\
&=n-c^\prime.\notag
\end{eqnarray}

There are at most \(\frac{4ln}{c}=\frac{n^\prime}{2}\) vertices of degree at least \(c\), and at most \(l^\prime c^\prime\) vertices not adjacent to all of \(A\). \[\frac{n-3l^2}{l_1+1}-\frac{n^\prime}{2}\geq\frac{n-3l^2}{2\left(l_1+1\right)}\geq 16l^4\binom{2l}{l}\geq l^\prime c^\prime,\] and so by the definition of \(B\), each vertex \(x\in B\) is adjacent to a vertex \(y\) which is adjacent to all of \(A\) and such that \(d_G(y)\leq c\). 

This condition on the vertices in \(B\) enables us to find, for each \(b\in B\), a copy of \(H^\prime\) from which half the vertices have small degree, and whose intersection with \(B\) contains \(b\). Further, we can find a set \(U\) of \(l^\prime-1\) vertices of degree at most \(c\) which are each adjacent to all of \(A\), so for any \(z\in A\), \(G\left[\left(U\cup\{x\}\cup\{y\}\cup\left(A\setminus\{z\}\right)\right)\right]\) is a graph on \(2l^\prime\) vertices which contains a copy of \(K_{l^\prime,l^\prime-1}\) with an extra vertex \(x\) adjacent to some vertex in the larger set. We can find a copy of \(H^\prime\) in this by letting a leaf of \(H^\prime\) correspond to \(x\), and so as in (\ref{cprime}), every vertex in \(B\) must have degree at least \(n-c^\prime\). If \(B\) contained at least \(l_1\) vertices, they would have common neighbourhood of size at least \(n-l_1c^\prime\geq l\), allowing us to find \(H_1\) in \(G[V(G)\setminus A]\), and again as the common neighbourhood of \(A\) is of size at least \(2l\), we can find a disjoint copy of \(H^\prime\), giving a copy of \(H\) in \(G\). Thus \(|B|\leq l_1-1\), and so \(|A\cup B|\leq l^\prime+l_1-1=l-1\).

We now define two more sets of vertices as follows:
\begin{align*}
D&=\left\{x\in G\setminus(A\cup B)\,|\,d_{A\cup B}(x)\geq l_1\right\},\\
E&=\left\{x\in G\setminus(A\cup B)\,|\,d_{A\cup B}(x)<l_1\right\}.
\end{align*}

We note that any vertex not in \(A\cup B\) which is adjacent to all of \(A\) is in \(D\), and thus \(\left|E\right|\leq l^\prime c^\prime\). There can be no \(H_1\) in \(E\), so the number of edges in \(E\) is at most \(\left(l_1-1\right)\left|E\right|\) by Erd\H{o}s-S\'os. No vertex \(v\in D\) can have a neighbour \(y\in D\cup E\), as we can find a set \(U\) of \(l_1-1\) vertices in \(A\cup B\) adjacent to \(v\), and \(W\subseteq (D\cup E)\setminus\{v,y\}\) consisting of \(l_1-1\) vertices adjacent to all of \(U\), and as before we can find a copy of \(H_1\) on \(U\cup W\cup\{v,y\}\) with only \(l_1-1\) vertices from \(A\cup B\); a contradiction. Thus all edges in \(G[D\cup E]\) are in \(E\).

Letting \(|A\cup B|=t\), we bound the number of edges in \(G\) by
\begin{align}
&\binom{t}{2}+t\left(n-t-\left|E\right|\right)+\left(l_1-1\right)\left|E\right|+\left(l_1-1\right)\left|E\right|\label{sec}\\
&=\binom{t}{2}+t\left(n-t\right)+\left(2l_1-2-t\right)\left|E\right|.\notag
\end{align}

If \(t<l-1\), then since \(|E|\leq l^\prime c^\prime\) the number of edges in \(G\) is at most \(\binom{t}{2}+t\left(n-t\right)+2l_1l^\prime c^\prime\leq\left(l-1\right)\left(n-l+1\right)\), for \(n\geq 2l^2c^\prime+l^2\).

The common neighbourhood of \(A\cup B\) has size at least \(n-(l-1)-(l-1)c^\prime\), as each vertex in \(A\cup B\) is adjacent to all but \(c^\prime\) vertices in \(G\). Thus we can find a copy of \(K_{l-1,n-(l-1)(c^\prime+1)}\subseteq G\), where the smaller class is \(A\cup B\). If \(H\) does not contain a perfect matching, then by Lemma \ref{noPerfLem} we can partition the vertices into unequal sets \(X\),\(Y\), the larger of which is empty, and the smaller of which contains one edge. This is clearly present in \(G\) if \(A\cup B\) contains an internal edge.

Counting all edges in \(G\), we see that by (\ref{sec}),
\[|E(G)|\leq(l-1)(n-l-1)-(l-2l_1+1)|E|+C_H,\]
where \(C_H=\binom{l-1}{2}\) if \(H\) admits a perfect matching, and \(C_H=0\) otherwise. As \(l_1\) is minimal, \((l-2l_1+1)>0\), and so the number of edges is maximized when \(|E|=0\).
\end{proof}

It is unlikely that the bound on \(n\) in Theorem \ref{treethm} is optimal. Determining the minimal value of \(n\) for which this construction is optimal remains an open question.

\section{Acknowledgements}
We would like to thank the anonymous referee for their suggestions on improvements to the style, structure, and content of this paper. We would also like to thank B\'ela Bollob\'as for pointing us towards this interesting problem.


\thebibliography{99}

\bibitem{dobson2}{S.~Balasubramanian, E.~Dobson, \emph{On the Erd\H{o}s-S\'os conjecture for graphs with no \(K_{2,s}\)}, J. Graph Theory \textbf{56} (2007), pp. 301-310.}
\bibitem{BGLS}{P.N.~Balister, E.~Gy\H{o}ri, J.~Lehel, R.H.~Schelp, \emph{Connected graphs without long paths}, Discrete Mathematics \textbf{308} (2008), pp. 4487--4494.}
\bibitem{mgt}{B.~Bollob\'as, \emph{Modern graph theory}, Springer-Verlag, New York (2002).}
\bibitem{dobson}{S.~Brandt, E.~Dobson, \emph{The Erd\H{o}s-S\'os conjecture for graphs of girth 5}, Discrete Mathematics \textbf{150} (1-3) (1996), pp. 411--414.}
\bibitem{EG}{P.~Erd\H{o}s, T.~Gallai, \emph{On maximal paths and circuits of graphs}, Acta Math. Acad. Sci. Hungar. \textbf{10} (1959), pp. 337--356.}
\bibitem{gorgol}{I.~Gorgol, \emph{Tur\'an Numbers for disjoint copies of graphs}, Graphs and Combinatorics (4 January 2011), pp. 1--7.}
\bibitem{Kopylov}{G.N. Kopylov, \emph{On maximal paths and cycles in a graph}, Soviet Math. Dokl. \textbf{18} (1977), pp. 593-596.}
\bibitem{mclennan}{A.~McLennan, \emph{The Erd\H{o}s-S\'os conjecture for trees of diameter four}, J. Graph Theory \textbf{49} (2005), pp. 291-301.}
\bibitem{mp}{W.~Moser, J.~Pach, \emph{Recent developments in combinatorial geometry}, in: New trends in discrete and computational geometry, Springer-Verlag, New York (1993).}
\bibitem{saclewozniak}{J.F.~Sacl\'e, M.~Wo\'zniak, \emph{The Erd\H{o}s-S\'os conjecture for graphs without \(C_4\)}, J. Combin. Theory Ser. B \textbf{70} (2) (1997), pp. 367--372.}
\bibitem{sidorenko}{A.~Sidorenko, \emph{Asymptotic solution for a new class of forbidden \(r\)-graphs}, Combinatorics \textbf{9} (2) (1989), pp. 207--215.}
\bibitem{simonovits}{M.~Simonovits, \emph{A method for solving extremal problems in extremal graph theory}, in: Theory of Graphs (P. Erd\H{o}s, G. Katona eds.), Academic Press, New York (1968), pp. 279--319.}
\bibitem{Turan1}{P.~Tur\'an, Egy gr\'afelm\'eleti sz\'els\H{o}\'ert\'ekfeladatr\'ol, Mat. es Fiz. Lapok. \textbf{48} (1941), pp. 436--452.}
\bibitem{Turan2}{P.~Tur\'an, On the theory of graphs, Colloquium Math. \textbf{3} (1954), pp. 19--30.}
\bibitem{wozniak}{M.~Wo\'zniak, \emph{On the Erd\H{o}s-S\'os conjecture}, J. Graph Theory \textbf{21} (2) (1996), pp. 229--234.}
\end{document}